\documentclass[12pt]{amsart}
\calclayout

\title[Limit of Bergman kernels on a tower of coverings]{Limit of Bergman kernels on a tower of coverings of compact K\"ahler manifolds}
\author[S. Yoo]{Sungmin Yoo}
\address{Center for Complex Geometry, Institute for Basic Science (IBS), Daejeon 34126, Republic of Korea}
\email{sungmin@ibs.re.kr}
\author[J. Yum]{Jihun Yum}
\address{Center for Complex Geometry, Institute for Basic Science (IBS), Daejeon 34126, Republic of Korea}
\email{jihun0224@ibs.re.kr}
\date{\today}

\usepackage{amsthm,amssymb,latexsym}

\usepackage[abbrev]{amsrefs}
\usepackage{hyperref}

\usepackage[autostyle]{csquotes}
\usepackage{xcolor}
\usepackage{cite}

\newcommand{\CC}{\mathbb{C}} 		
\newcommand{\NN}{\mathbb{N}} 		
\newcommand{\PP}{\mathbb{P}} 		
\newcommand{\re}{\mathrm{Re}} 		
\newcommand{\Aut}{\mathrm{Aut}} 	
\newcommand{\IM}{\sqrt{-1}}		 	

\newcommand{\Bergman}{\mathcal{B}} 

\newcommand{\norm}[1]{\left\|#1\right\|}                
\newcommand{\inner}[1]{\left\langle{#1}\right\rangle}   

\newtheorem*{Question}{Question}

\theoremstyle{plain}
\newtheorem{thm}{Theorem}[section]
\newtheorem{lem}[thm]{Lemma}
\newtheorem{prop}[thm]{Proposition}
\newtheorem{cor}[thm]{Corollary}

\newtheorem{mainthm}{Theorem} 
\newtheorem{maincor}[mainthm]{Corollary}

\theoremstyle{definition}
\newtheorem{defn}[thm]{Definition}

\theoremstyle{remark}
\newtheorem{rmk}[thm]{Remark}

\numberwithin{equation}{section}

\begin{document}

\maketitle

\begin{abstract}
	We prove the convergence of the Bergman kernels and the $L^2$-Hodge numbers on a tower of Galois coverings $\{X_j\}$ of a compact K\"ahler manifold $X$ converging to an infinite Galois (not necessarily universal) covering $\widetilde{X}$.
	We also show that, as an application, sections of canonical line bundle $K_{X_j}$ for sufficiently large $j$ give rise to an immersion into some projective space, if so do sections of $K_{\widetilde{X}}$.
\end{abstract}

\section{\bf Introduction}


Let $\{X_j\}$ be a tower of Galois coverings of a complex manifold $X$ converging to an infinite Galois covering $\widetilde{X}$ of $X$. A tower of coverings is said to be {\it Bergman stable} if the pull-back of the Bergman kernels of $X_j$ converges locally uniformly to that of $\widetilde{X}$ as $j \rightarrow \infty$.
A famous theorem by Kazhdan \cite{kazhdan83} states that a tower of coverings of compact Riemann surfaces with genus $g \ge 2$ converging to {\it the universal covering} is Bergman stable (see McMullen\cite{mcmullen2013entropy} and Rhodes\cite{rhodes1993sequences} for other proofs).
We remark that in contrast, T. Ohsawa \cite{Ohsawa10} constructed an example of a tower of {\it non}-Galois coverings which is not Bergman stable. 

Recently, Baik, Shokrieh and Wu generalized Kazhdan's theorem where $\widetilde{X}$ is {\it any infinite Galois covering} of a compact Riemann surface $X$ (see Theorem A in \cite{BaikShokriehWu20}).
The purpose of this work is to prove the following generalization for higher dimensional complex manifolds.

\begin{mainthm}[Theorem \ref{thm: Bergman stable}]     \label{thm A} 
    Any tower of Galois coverings of a compact K\"ahler manifold converging to an infinite Galois covering is Bergman stable. 
\end{mainthm}

Chen and Fu also proved this type of theorem when the base manifold $X$ is possibly non-compact under the additional assumption that both top and base manifolds satisfy a certain type of K\"{a}hler-hyperbolicity condition (see Theorem 1.3 in \cite{ChenFu16}). 
They used Donnelly-Fefferman type $L^2$-estimates for the $\overline{\partial}$-Laplacian.
We emphasize that our results do not require any further assumption except that the base manifold is compact K\"ahler. 

Instead of $\overline{\partial}$-estimates, we will take a topological approach as in \cite{BaikShokriehWu20}.
Note that the Bergman stability is equivalent to the convergence of the $L^2$-Hodge numbers $h^{n,0}_{(2)}$, the von Neumann dimensions of the spaces of $L^2$-harmonic (holomorphic) $(n,0)$-forms, which is called the {\it Kazhdan equality}.
This fact was first observed by Kazhdan (cf. \cite{kazhdan83,ChenFu16}, see also Proposition \ref{prop: uniform convergence of Bergman kernels}).
If $X$ is a Riemann surface, it is enough to show the convergence of the $L^2$-Betti numbers $b_1^{(2)}$ since $b_1^{(2)}=2h^{1,0}_{(2)}$ by the Hodge decomposition.
In fact, Baik, Shokrieh and Wu proved this convergence by a variant of {\it L\"{u}ck's approximation theorem} (see Theorem 4.6 in \cite{BaikShokriehWu20}).

However, for the higher dimensional case $n>1$, the convergence of the $L^2$-Betti numbers does not directly imply that of the $L^2$-Hodge numbers $h^{n,0}_{(2)}$.
Using the theory of von Neumann algebras, we prove the following general version.

\begin{mainthm}[Kazhdan's equality for the $L^2$-Hodge numbers] \label{thm B}
Let $\{X_j\rightarrow X\}$ be a tower of Galois coverings of a compact K\"{a}hler manifold $X$ converging to an infinite Galois covering $\widetilde{X}\rightarrow X$. 
Then, for any $0\leq p,q \leq n$, we have
$$
   h^{p,q}_{(2)}(X_j) \left( = \frac{h^{p,q}(X_j)}{|\mathbf{G}_j|} \right) \longrightarrow   h^{p,q}_{(2)}(\widetilde{X}) 
$$
as $j \rightarrow \infty$, where $\mathbf{G}_j$ is the group of deck transformations of the covering $X_j\rightarrow X$, and $h^{p,q}(X_j)$ is the (ordinary) Hodge number of $X_j$.
\end{mainthm}

On the other hand, the Bergman kernel and metric are strongly related to the projectivity of a manifold in the following sense:
For a complex manifold $X$, one can consider a holomorphic map $\iota_X$ from $X$ into the (possibly infinite dimensional) complex projective space $\PP(H^0_{(2)}(X, K_{X})^*)$ using an orthonormal basis of the Hilbert space $H^0_{(2)}(X, K_{X})$.
The map $\iota_X$ is called the {\it Kobayashi map} (\cite{Kobayashi59}), which can be regarded as a $L^2$-version of the Kodaira map.
When $\iota_X$ is an embedding, we say that $K_X$ is {\it very ample}.
Then the Bergman kernel of $X$ is the pull-back of the Fubini-Study metric, which is a hermitian metric on $\mathcal{O}(1)$-bundle over $\PP(H^0_{(2)}(X, K_{X})^*)$, and the Bergman metric of $X$ is its curvature form.
Now, we consider the following natural question.

\begin{Question}
Assume that the canonical line bundle $K_{\widetilde{X}}$ of top manifold $\widetilde{X}$ is very ample. Is it true that the canonical line bundle $K_{X_j}$ is also very ample for sufficiently large $j$?
\end{Question}

As an application of Theorem \ref{thm A}, we give a partial answer to the above question as follows.

\begin{maincor} \label{thm C}
    Let $\{X_j\rightarrow X\}$ be a tower of Galois coverings of a compact K\"{a}hler manifold $X$ converging to an infinite Galois covering $\widetilde{X}\rightarrow X$. 
    Assume that $\iota_{\widetilde{X}} : \widetilde{X} \rightarrow \PP(H^0_{(2)}(\widetilde{X}, K_{\widetilde{X}})^*)$ is a holomorphic immersion.
    Then there exists a positive integer $N \in \NN$ such that for all $j \ge N$, $\iota_{X_j} : X_j \rightarrow \PP(H^0(X_j, K_{X_j})^*)$ is holomorphic immersion.
\end{maincor}

In fact, S.-K. Yeung obtained some affirmative results assuming that the top manifold $\widetilde{X}$ is a Hermitian symmetric manifold of non-compact type, which is simply-connected; when the base manifold $X$ is compact he \cite{Yeung00} proved the very ampleness of $K_{X_j}$ (see the last paragraph in Section \ref{sec 6}), and when $X$ is non-compact he \cite{Yeung12} proved the existence of a holomorphic immersion $X_j \rightarrow \PP(H^0(X_j, K_{X_j})^*)$.
\vskip 0.5em
This paper is organized as follows. In Section \ref{sec 2}, we briefly review relevant materials from the theory of von Neumann algebras and results on the approximation theorem of $L^2$-Betti numbers. 
In Section \ref{sec 3}, we discuss their applications to a tower of coverings of compact Riemannian manifolds in terms of the Schwartz kernels.
In Section \ref{sec 4}, we prove Theorem \ref{thm B} using the properties of the Schwartz kernels.
In Section \ref{sec 5}, we review the definitions of Bergman kernel and Bergman metric and establish some of their basic properties, and prove Theorem \ref{thm A}.
Corollary \ref{thm C} is proved in Section \ref{sec 6}.

\vskip 0.5em

\noindent
\textbf{Acknowledgements}. 
The authors would like to thank Professor Jun-Muk Hwang for his suggestion of this work and valuable comments.
This work was supported by the Institute for Basic Science (IBS-R032-D1).


\section{Preliminaries} \label{sec 2}
In this section, we briefly review some terminologies, notations, and basics on the theory of von Neumann algebras.
For more details, we refer to \cite{Luck02,kammeyer2019introduction}.
We also discuss the approximation theorem of $L^2$-Betti numbers by L\"{u}ck \cite{Luck94} and Baik-Shokrieh-Wu \cite{BaikShokriehWu20}.

\subsection{Tower of coverings}

Let $X$ be a topological space (or differentiable or complex manifold).
Recall that $\phi:Y\rightarrow X$ is called {\it Galois} (or normal) if the group of deck transformations 
$\mathbf{G} := \{ f \in \Aut(Y) : \phi \circ f = \phi \} $ acts on $\phi^{-1}(p)$ transitively for all $p \in X$, and we may describe a Galois covering as a quotient map $Y\rightarrow Y/\mathbf{G}$.

\begin{defn}
    A sequence of topological spaces (or differentiable or complex manifolds) $\{X_j\}_{j=1}$ is called a {\it tower of coverings} of $X$ if there exists a sequence of continuous (or smooth or holomorphic, resp.) Galois covering maps 
    $\{\phi_j : X_j \rightarrow X\}_{j=1}$ satisfying 
    $$\pi_1(X) \supset \pi_1(X_1) \supset \pi_1(X_2) \supset \cdots$$
    such that $\pi_1(X_j)$ is a normal subgroup of $\pi_1(X)$ of finite index $[ \pi_1(X) : \pi_1(X_j) ]$ for each $j$, where $\pi_1(X)$ denotes the first fundamental group of $X$.
\end{defn}

\begin{defn}
We say that a tower of (Galois) coverings $\{\phi_j : X_j \rightarrow X\}$ {\it converges} to an infinite Galois covering $\phi : \widetilde{X} \rightarrow X$ if it satisfies
$$
\bigcap_{j=1} \pi_1(X_j) = \pi_1(\widetilde{X}).
$$
We will refer to $\widetilde{X}$ as the {\it top manifold} and $X$ as the {\it base manifold}.
\end{defn}

\begin{rmk}
    There exists an another sequence of Galois coverings $\{\widetilde{\phi}_j : \widetilde{X} \rightarrow X_j\}$ such that $\phi = \phi_j \circ \widetilde{\phi}_j$ because $\pi_1(\widetilde{X})$ is a normal subgroup of $\pi_1(X_j)$.
\end{rmk}

\subsection{Hilbert $\mathbf{G}$-modules and von Neumann dimension}

Let $\mathbf{G}$ be a discrete group.
Define the Hilbert space $\ell^2(\mathbf{G})$ by 
$$
\ell^2(\mathbf{G}):=\Big\{ u:\mathbf{G}\rightarrow\mathbb{C}\ \Big| \sum_{g\in\mathbf{G}}|u(g)|^2<\infty     \Big\}.
$$
Then the set of indicator functions $\{\delta_g\}_{g\in\mathbf{G}}$ is an orthonormal basis of $\ell^2(\mathbf{G})$.

\begin{defn}
Let $V$ be a Hilbert space and let $\mathbf{G}$ be a discrete group.
\begin{itemize}
    \item $V$ is called a {\it Hilbert $\mathbf{G}$-module} if it admits a left unitary action of $\mathbf{G}$.
    \item $V$ is called a {\it free Hilbert $\mathbf{G}$-module} if it is a Hilbert $\mathbf{G}$-module and $\mathbf{G}$-equivariant unitary isomorphic to $\ell^2(\mathbf{G})\otimes H$ for some Hilbert space $H$.
    \item $V$ is called a {\it projective Hilbert $\mathbf{G}$-module} if it is a Hilbert $\mathbf{G}$-module and admits a unitary isometric $\mathbf{G}$-equivariant embedding into a free Hilbert $\mathbf{G}$-module.
\end{itemize}
\end{defn}

Let $V$ be a projective Hilbert $\mathbf{G}$-module (embedded into some free Hilbert $\mathbf{G}$-module $\ell^2(\mathbf{G})\otimes H$),
and $\pi_V$ be the orthogonal projection from $\ell^2(\mathbf{G})\otimes H$ onto $V$.

\begin{defn}
The {\it von Neumann dimension} (or {\it $\mathbf{G}$-dimension}) of $V$ is defined by
$$
{\rm dim}_{\mathbf{G}}(V):={\rm Tr}_{\mathbf{G}}(\pi_V)=\sum_{\alpha\in J}\langle \pi_V(\delta_e\otimes u_\alpha),\delta_e\otimes u_\alpha\rangle,
$$
where $\{u_{\alpha}\}_{\alpha\in J}$ is an orthonormal basis of $H$ and $e$ is the identity element of $\mathbf{G}$.
\end{defn}

\begin{rmk}\label{rmk: finite von Neumann}
This is a well-defined invariant (independent of the choice of the $\mathbf{G}$-equivariant embedding).
If ${\rm dim}_{\mathbb{C}}(V)$ and $|\mathbf{G}|$ are both finite, ${\rm dim}_{\mathbf{G}}(V)=\frac{1}{|\mathbf{G}|}{\rm dim}_{\mathbb{C}}(V).$
\end{rmk}

\subsection{$L^2$-Betti numbers and the approximation theorem}
Let $X$ be a free $\mathbf{G}$-CW-complex of finite type for some discrete group $\mathbf{G}$.
\begin{defn}
Let $(C_{\ast}(X),d_{\ast})$ be the cellular chain complex of $X$.
Define the (cellular) {\it $L^2$-chain complex} and {\it $L^2$-cochain complex} of $X$ by
$$
(C_{\ast}^{(2)}(X),d_{\ast}^{(2)}):=(\ell^2(\mathbf{G})\otimes_{\mathbb{Z}\mathbf{G}}C_{\ast}(X),{\rm id}\otimes d_{\ast}),
$$
$$
(C^{\ast}_{(2)}(X),d^{\ast}_{(2)}):=({\rm Hom}_{\mathbb{Z}\mathbf{G}}(C_{\ast}(X),\ell^2(\mathbf{G})),d^{\ast}_{(2)}).
$$
\end{defn}
The above complexes are sequences of (finitely generated) projective Hilbert $\mathbf{G}$-modules.
Therefore the followings inherit the structure of projective Hilbert $\mathbf{G}$-modules.
\begin{defn}
The $k$-th reduced {\it $L^2$-homology} and {\it $L^2$-cohomology} of the pair $(X,\mathbf{G})$ are projective Hilbert $\mathbf{G}$-modules, defined by
$$
H_k^{(2)}(X;\mathbf{G}):={\rm Ker}\ d_k^{(2)}/\overline{{\rm Im}\ d_{k+1}^{(2)}},\ \ \ 
H^k_{(2)}(X;\mathbf{G}):={\rm Ker}\ d^k_{(2)}/\overline{{\rm Im}\ d^{k-1}_{(2)}}.
$$
The $k$-th homological and cohomological {\it $L^2$-Betti number} of $(X,\mathbf{G})$ are defined by
$$
b_k^{(2)}(X;\mathbf{G}):={\rm dim}_{\mathbf{G}}H_k^{(2)}(X;\mathbf{G}),\ \ \ 
b^k_{(2)}(X;\mathbf{G}):={\rm dim}_{\mathbf{G}}H^k_{(2)}(X;\mathbf{G}).
$$
\end{defn}

\begin{rmk}
It is well-known that $b^k_{(2)}(X;\mathbf{G})=b_k^{(2)}(X;\mathbf{G})$ (cf. Theorem 3.24 in \cite{kammeyer2019introduction}).
    By Remark \ref{rmk: finite von Neumann}, if $|\mathbf{G}|$ is finite, we have
    $$
    b^k_{(2)}(X;\mathbf{G})=b_k^{(2)}(X;\mathbf{G})=\frac{1}{|\mathbf{G}|}b_k(X),
    $$
    where $b_k(X)$ is the (ordinary) $k$-th Betti number of $X$.
\end{rmk}

Note that a Galois covering $Y$ of a finite CW complex $X$ is a free $\mathbf{G}$-CW-complex of finite type with the group of deck transformations $\mathbf{G}:=\pi_1(X)/\pi_1(Y)$.
The following approximation theorem was first proved by L\"{u}ck when $\widetilde{X}$ is the universal covering of $X$ (see Theorem 0.1 in \cite{Luck94}).
Baik, Shokrieh and Wu generalized this to any infinite Galois covering in the following form.

\begin{thm}[cf. Theorem 4.6 in \cite{BaikShokriehWu20}] \label{thm: top Kazhdan eq}
Let $\{X_j\}$ be a tower of coverings of a finite connected CW complex $X$ converging to an infinite Galois covering $ \widetilde{X} \rightarrow X$. Then
$$
b^k_{(2)}(\widetilde{X};\mathbf{G})=\lim\limits_{j\rightarrow\infty}b^k_{(2)}(X_j;\mathbf{G}_j)=\lim\limits_{j\rightarrow\infty}\frac{b_k(X_j)}{|\mathbf{G}_j|},
$$
where $\mathbf{G}:=\pi_1(X)/\pi_1(\widetilde{X})$, $\mathbf{G}_j:=\pi_1(X)/\pi_1(X_j)$, and $|\mathbf{G}_j|:=[\pi_1(X):\pi_1(X_j)]$.
\end{thm}

\section{Tower of coverings of compact Riemannian manifolds} \label{sec 3}
In this section, we will explain the convergence of the trace forms of Schwartz kernels on a tower of coverings of compact Riemannian manifolds via the convergence of $L^2$-Betti numbers (without using the heat kernel method).
\subsection{Kazhdan's equality for harmonic spaces}

Let $(X,g)$ be a (possibly non-compact) Riemannian manifold.
Let $A^k(X)$ be the space of complex differential $k$-forms on $X$, i.e.,
the set of smooth sections of the bundle $\bigwedge^k T^{\ast}_{\mathbb{C}}X$, where $T^{\ast}_{\mathbb{C}}X:=T^{\ast}_{\mathbb{R}}X\otimes\mathbb{C}$ is the complexified cotangent bundle.
For any compactly supported forms $u,v\in A^k_c(X)$, define the inner product by
$$
\langle u,v\rangle:=\int_X u\wedge\ast\overline{ v}=\int_X(u,v)\ dV_g,
$$
where $\ast$ is the Hodge star operator and $dV_g$ is the volume form of $g$.
Denote its Hilbert space completion by $A^k_{(2)}(X)$.
Then the space of $L^2$ harmonic (complex differential) $k$-forms is defined by
$$
\mathcal{H}^k_{(2)}(X):=\{u\in A^k_{(2)}(X)\ |\ \Delta_{g} u=0\}
$$
(we also use the notation $\mathcal{H}^k(X,\mathbb{C})$ without the subscript when $X$ is compact).

Let $Y\rightarrow X$ be a Galois covering of a compact Riemannian manifold $(X,g)$, with the group of deck transformations $\mathbf{G}=\pi_1(X)/\pi_1(Y)$.
Since $X$ admits a finite CW complex structure, we can define the reduced $L^2$-cohomology $H^k_{(2)}(Y;\mathbf{G})$ of $(Y,\mathbf{G})$.
On the other hand, $Y$ admits a $\mathbf{G}$-invariant complete Riemannian metric $\widetilde{g}$, induced by the metric $g$ of $X$.
In \cite{Dodziuk77}, Dodziuk showed that $\mathcal{H}^k_{(2)}(Y)$ also has a structure of projective Hilbert $\mathbf{G}$-module and
proved the following $L^2$-version of the Hodge-de Rham theorem.

\begin{thm}[cf. Theorem 1 in \cite{Dodziuk77}] \label{thm: L2 hodge-de rham}
We have the following canonical isomorphisms between projective Hilbert $\mathbf{G}$-modules:
$$
\mathcal{H}^k_{(2)}(Y)\simeq H^k_{\rm dR (2)}(Y,\mathbb{C})\simeq H^k_{(2)}(Y;\mathbf{G}),
$$
where $H^k_{\rm dR (2)}(Y,\mathbb{C})$ is the reduced $L^2$-de Rham cohomology.
In particular, we have 
$$
{\rm dim}_{\mathbf{G}}\mathcal{H}^k_{(2)}(Y)={\rm dim}_{\mathbf{G}}H^k_{(2)}(Y;\mathbf{G}).
$$
\end{thm}

Note that as in Remark \ref{rmk: finite von Neumann}, if $Y$ is compact and $|\mathbf{G}|$ is finite, we have
$$
{\rm dim}_{\mathbf{G}}\mathcal{H}^k_{(2)}(Y)
=\frac{{\rm dim}_{\mathbb{C}}\mathcal{H}^k_{(2)}(Y)}{|\mathbf{G}|}
=\frac{{\rm dim}_{\mathbb{C}}\mathcal{H}^k(Y,\mathbb{C})}{|\mathbf{G}|}.
$$
Then, Theorem \ref{thm: L2 hodge-de rham} and Theorem \ref{thm: top Kazhdan eq} imply Kazhdan's equality for harmonic spaces.

\begin{cor} \label{cor: Riemann hodge Kazhdan eq}
Let $\{X_j\}$ be a tower of coverings of a compact Riemannian manifold $X$ converging to an infinite Galois covering $ \widetilde{X} \rightarrow X$. Then
$$
{\rm dim}_{\mathbf{G}}\mathcal{H}^k_{(2)}(\widetilde{X})
=\lim\limits_{j\rightarrow\infty}{\rm dim}_{\mathbf{G}_j}\mathcal{H}^k_{(2)}(X_j)
=\lim\limits_{j\rightarrow\infty}\frac{{\rm dim}_{\mathbb{C}}\mathcal{H}^k(X_j,\mathbb{C})}{|\mathbf{G}_j|},
$$
where $\mathbf{G}:=\pi_1(X)/\pi_1(\widetilde{X})$, $\mathbf{G}_j:=\pi_1(X)/\pi_1(X_j)$, and $|\mathbf{G}_j|=[\pi_1(X):\pi_1(X_j)]$.
\end{cor}

\subsection{Convergence of the trace forms of Schwartz kernels}
Let $X$ be a Riemannian manifold.
Consider the orthogonal projection operator $\pi^k_{X}$ from $A^k_{(2)}(X)$ onto
$\mathcal{H}^k_{(2)}(X)$.
Let $\{e_{\alpha}\}$ be orthonormal bases for $A^k_{(2)}(X)$.
\begin{defn}
The {\it Schwartz kernel} form of the operator $\pi^k_{X}$ is a symmetric smooth double form $\mathcal{K}^k_X\in A^k_{(2)}(X)\otimes A^k_{(2)}(X)$ defined by for $x,y\in X$,
$$
\mathcal{K}^k_X(x,y)
:=\sum_{\alpha,\beta}\inner{\pi^k_{X}(e_{\alpha}),e_{\beta}}e_{\alpha}(x)\otimes \overline{e_{\beta}(y)}.
$$
The {\it trace form} of the Schwartz kernel ${\rm tr}\mathcal{K}^k_X$ is a volume form on $X$ defined by
$$
{\rm tr}\mathcal{K}^k_X(x)
=\sum_{\alpha}\inner{\pi^k_{X}(e_{\alpha}),e_{\alpha}}e_{\alpha}(x)\wedge\ast\overline{e_{\alpha}(x)}.
$$
\end{defn}

\begin{rmk}
One can easily show that the above definitions are independent of the choice of the orthonormal basis $\{e_{\alpha}\}$.
Moreover, the Schwartz kernel and its trace form can be represented as:
    $$
    \mathcal{K}^k_X(x,y)=\sum_{\alpha} u_{\alpha}(x)\otimes \overline{u_{\alpha}(y)}
    {\ \ \ \ \rm and \ \ \ \ }
    {\rm tr}\mathcal{K}^k_X=\sum_{\alpha} u_{\alpha}\wedge\ast\overline{u_{\alpha}},
    $$
where $\{u_{\alpha}\}_{\alpha}$ is an orthonormal basis for the Hilbert space $\mathcal{H}^k_{(2)}(X)$.
\end{rmk}
The name of the Schwartz kernel comes from the property that for all $u\in A^k_{(2)}(X)$,
$$
\pi^k_{X}(u)(x)
=\inner{u(\cdot),\mathcal{K}^k_X(\cdot,x)}_{X}
=\int_{X} u(y)\wedge\ast\overline{\mathcal{K}^k_X(y,x)},
$$
where $\overline{ \mathcal{K}^k_X(y,x)}=\sum_{\alpha,\beta}\inner{\pi^k_{X}(e_{\alpha}),e_{\beta}} \overline{e_{\beta}(y)}\otimes e_{\alpha}(x)$.
Moreover, for all $u\in \mathcal{H}^k_{(2)}(X)\subset A^k_{(2)}(X)$, we have the following {\it reproducing property}:
$$
u(x)
=\int_{X} u(y)\wedge\ast\overline{\mathcal{K}^k_X(y,x)}.
$$

Note that if $X$ is compact, then
$$
{\rm dim}_{\mathbb{C}}(\mathcal{H}^k(X))
={\rm Tr}(\pi^k_{X}):=\sum_{\alpha}\inner{\pi^k_{X}(e_{\alpha}),e_{\alpha}}
=\int_X{\rm tr}\mathcal{K}^k_X.
$$

Consider a Galois covering $\phi:Y\rightarrow X$ of a compact Riemannian manifold $X$ with the group of deck transformations $\mathbf{G}=\pi_1(X)/\pi_1(Y)$.
Let $F$ be a {\it fundamental domain}, i.e., an open subset $F\subset Y$ satisfying
\begin{itemize}
    \item[(1)] $Y=\bigcup_{g\in\mathbf{G}}g(\overline{F})$,
    \item[(2)] $g(F)\cap h(F)=\emptyset$ for all $g\neq h\in \mathbf{G}$,
    \item[(3)] $\overline{F}\setminus F$ has zero measure.
\end{itemize}
Note that $A^k_{(2)}(Y)\simeq\ell^2(\mathbf{G})\otimes A^k_{(2)}(F)$ is a free Hilbert $\mathbf{G}$-module by the isomorphism
$$
u\simeq \sum_{g\in\mathbf{G}}\delta_g\otimes g^{*}(u|_{g(F)}),
$$
where $g^{\ast}(u|_{g(F)})$ is the pull-back of the restriction form $u|_{g(F)}$ on $g(F)$ to $F$ by the $g$-action.
Then $\mathcal{H}^k_{(2)}(Y)$ is a projective Hilbert $\mathbf{G}$-module (embedded into $A^k_{(2)}(Y)$) so that
$$
{\rm dim}_{\mathbf{G}}(\mathcal{H}^k_{(2)}(Y))
:=
{\rm Tr}_{\mathbf{G}}(\pi^k_{Y})=\int_F{\rm tr}\mathcal{K}^k_Y.
$$
If $Y$ is compact and $|\mathbf{G}|$ is finite, Remark \ref{rmk: finite von Neumann} implies that
$$
{\rm dim}_{\mathbf{G}}(\mathcal{H}^k_{(2)}(Y))
=\frac{1}{|\mathbf{G}|}{\rm dim}_{\mathbb{C}}(\mathcal{H}^k(Y))
=\frac{1}{|\mathbf{G}|}\int_Y{\rm tr}\mathcal{K}^k_Y.
$$
Therefore, the Kazhdan equality for harmonic spaces (Corollary \ref{cor: Riemann hodge Kazhdan eq}) implies that

\begin{cor} \label{cor: Riemann Schwartz Kazhdan eq}
Let $\{\phi_j:X_j\rightarrow X\}$ be a tower of coverings of a compact Riemannian manifold $X$ converging to an infinite Galois covering $\phi:\widetilde{X}\rightarrow X$. Then
$$
\int_{\widetilde{F}} {\rm tr}\mathcal{K}^k_{\widetilde{X}}
=\lim\limits_{j\rightarrow\infty}\int_{F_j} {\rm tr}\mathcal{K}^k_{X_j}
=\lim\limits_{j\rightarrow\infty}\frac{1}{|\mathbf{G}_j|}\int_{X_j}{\rm tr}\mathcal{K}^k_{X_j},
$$
where $\widetilde{F}$ and $F_j$ are fundamental domains of $\phi:\widetilde{X}\rightarrow X$ and $\phi_j:X_j\rightarrow X$, respectively.
\end{cor}

\begin{rmk}
    Note that the trace form of the Schwartz kernel is invariant under $\mathbf{G}$-action so that its push-forward is well-defined.
    Therefore, the above corollary can be written as
    $$
\int_X \phi_{\ast}({\rm tr}\mathcal{K}^k_{\widetilde{X}})
=\lim\limits_{j\rightarrow\infty}\int_{X} \phi_{j\ast}({\rm tr}\mathcal{K}^k_{X_j}).
$$
\end{rmk}

\begin{rmk}
    The above Corollary was proved by Cheeger and Gromov \cite{cheeger1985characteristic,cheeger1985bounds} using the heat kernel method when $\widetilde{X}$ is the universal covering and $X_j$ are (possibly non-compact) manifolds with finite volume and bounded geometry.
\end{rmk}

    

\subsection{Convergence of canonical forms}

In \cite{BaikShokriehWu20}, Baik, Shokrieh and Wu introduced the concept of the canonical measures for Riemann surfaces (see also \cite{shokrieh2019canonical}).
Similarly, for higher dimensional Riemannian manifolds, we can define the following:

\begin{defn}
Let $X$ be a Riemannian manifold and
let $\{u_{\alpha}\}_{\alpha\in J}$ be an orthonormal basis for $\mathcal{H}^k_{(2)}(X)$.
The {\it canonical measure} (of degree $k$) on $X$ is defined by
$$
\mu^k_{X}(U):=\sum\limits_{\alpha\in J}\int_U u_\alpha\wedge\ast\overline{u_\alpha}\ \ 
\left(=\int_U {\rm tr}\mathcal{K}^k_{X}\right)
$$
for any Borel subset $U\subset X$.
\end{defn} 

Consider a Galois covering $\phi:Y\rightarrow X$ of a compact Riemannian manifold $X$ with the group of deck transformations $\mathbf{G}=\pi_1(X)/\pi_1(Y)$.
Since the canonical measure $\mu^k_{Y}$ is invariant under isometries, its push-forward to $X$ is well-defined.
Denote it by $\mu^k_{\phi}$.
Then we have
$$
\mu^k_{\phi}(X)={\rm Tr}_{\mathbf{G}}(\pi^k_Y)
    ={\rm dim}_{\mathbf{G}}(\mathcal{H}^k_{(2)}(Y)).
$$
Moreover, if it is finite covering, then $\mu^k_{\phi}(X)=\mu^k_{Y}(Y)/|\mathbf{G}|$.
Therefore, Theorem 5.3 in \cite{BaikShokriehWu20} can be generalized to higher dimensional cases in the following form.

\begin{cor}
Let $\{X_j\}$ be a tower of coverings of a compact Riemannian manifold $X$ converging to an infinite Galois covering $ \widetilde{X} \rightarrow X$.
Then for any Borel subset $U\subset X$,
$$
\lim_{j\rightarrow\infty}\mu^k_{\phi_j}(U)=\mu^k_{\phi}(U).
$$
\end{cor}


\section{Tower of coverings of compact K\"{a}hler manifolds} \label{sec 4}
In this section, we will prove Theorem \ref{thm B}. 
More precisely, for a tower of coverings of compact K\"{a}hler manifolds, we will show the convergence of $L^2$-Hodge numbers using properties of the Schwartz kernels.

\subsection{Schwartz kernels of K\"ahler manifolds}
Let $(X,g)$ be a (possibly non-compact) K\"{a}hler manifold of dimension $n$.
Using the decomposition
$
T^{\ast}_{\mathbb{C}}X=T^{\ast'}X\oplus T^{\ast''}X,
$
we have
$$
A_{(2)}^{k}(X)=\bigoplus\limits_{p+q=k}A_{(2)}^{p,q}(X),
$$
where $A_{(2)}^{p,q}(X)$ is the Hibert space completion of complex differential $(p,q)$-forms on $X$, i.e.
smooth sections of the bundle $\bigwedge^p T^{\ast'}X\otimes\bigwedge^q T^{\ast''}X$.
Moreover, the densely defined operator
$
\overline{\partial}:A_{(2)}^{p,q}(X)\rightarrow A_{(2)}^{p,q+1}(X)
$
gives us the $\overline{\partial}$-Laplacian operator
$
\Delta_{\overline{{\partial}}}:=\overline{\partial}\overline{\partial}^{\ast}+\overline{\partial}^{\ast}\overline{\partial}.
$
Then the space of $L^2$-harmonic $(p,q)$-forms is defined by
$$
\mathcal{H}^{p,q}_{(2)}(X):=\{u\in A^{p,q}_{(2)}(X)\ |\ \Delta_{\overline{{\partial}}} u=0\}.
$$

Denote the projection operator by $\pi^{p,q}_X:A^{p,q}_{(2)}(X)\rightarrow\mathcal{H}^{p,q}_{(2)}(X)$.
Let $\{e_{\alpha}\}$ and $\{u_{\alpha}\}$ be orthonormal bases for $A^{p,q}_{(2)}(X)$ and $\mathcal{H}^{p,q}_{(2)}(X)$, respectively.

\begin{defn} \label{def: Schwartz kernel of (p,q)}
The {\it Schwartz kernel} $\mathcal{K}^{p,q}_{X}$ of $\pi^{p,q}_X$ is defined by
$$
\mathcal{K}^{p,q}_{X}(x,y)
:=\sum_{\alpha,\beta}\inner{\pi^{p,q}_{X}(e_{\alpha}),e_{\beta}}e_{\alpha}(x)\otimes \overline{e_{\beta}(y)}
=\sum_{\alpha} u_{\alpha}(x)\otimes \overline{u_{\alpha}(y)}.
$$
The {\it trace form} ${\rm tr}\mathcal{K}^{p,q}_{X}$ of the Schwartz kernel $\mathcal{K}^{p,q}_{X}$ is defined by
$$
{\rm tr}\mathcal{K}^{p,q}_{X}
:=\sum_{\alpha}\inner{\pi^{p,q}_{X}(e_{\alpha}),e_{\alpha}}e_{\alpha}\wedge\ast\overline{e_{\alpha}}
=\sum_{\alpha} u_{\alpha}\wedge\ast\overline{u_{\alpha}}.
$$
\end{defn}
As in the case of Riemannian manifolds, the above definitions are independent of the choice of the orthonormal basis and it satisfies that for all $u\in A^{p,q}_{(2)}(X)$, we have 
$$
\pi^{p,q}_{X}(u)(x)
=\int_{X} u(y)\wedge\ast\overline{\mathcal{K}^{p,q}_X(y,x)}.
$$
Therefore, for all $u\in \mathcal{H}^{p,q}_{(2)}(X)$, we have the {\it reproducing property}:
$$
u(x)
=\int_{X} u(y)\wedge\ast\overline{\mathcal{K}^{p,q}_X(y,x)}.
$$


\begin{rmk}
When $p=n$ and $q=0$, the trace form $\mathcal{B}_{X}:={\rm tr}\mathcal{K}^{n,0}_{X}$ is independent of the K\"{a}hler metric. 
This is called the (diagonal) {\it Bergman kernel form} of $X$.
\end{rmk}

Note that as in the case of Riemmannian manifolds, for a Galois covering $Y$ of a compact K\"{a}hler manifold $X$,
$\mathcal{H}^{p,q}_{(2)}(Y)$ is a projective Hilbert $\mathbf{G}$-module which is embedded into the free Hilbert $\mathbf{G}$-module $A^{p,q}_{(2)}(Y)$.
Thus for a fundamental domain $F$, we have
$$
{\rm dim}_{\mathbf{G}}(\mathcal{H}^{p,q}_{(2)}(Y))
:=
{\rm Tr}_{\mathbf{G}}(\pi^{p,q}_{Y})=\int_F{\rm tr}\mathcal{K}^{p,q}_Y.
$$
Remark \ref{rmk: finite von Neumann} implies that if $Y$ is compact and $|\mathbf{G}|$ is finite, then we have
$$
{\rm dim}_{\mathbf{G}}(\mathcal{H}^{p,q}_{(2)}(Y))
=\frac{1}{|\mathbf{G}|}{\rm dim}_{\mathbb{C}}(\mathcal{H}^{p,q}(Y))
=\frac{1}{|\mathbf{G}|}\int_Y{\rm tr}\mathcal{K}^{p,q}_Y.
$$

\begin{defn}
Let $Y\rightarrow X$ be a Galois covering of a compact K\"{a}hler manifold $X$, with the group of deck transformations $\mathbf{G}=\pi_1(X)/\pi_1(Y)$.
The $L^2$-{\it Hodge number} (of bi-degree $(p,q)$) for $Y$ is defined by
$$
h^{p,q}_{(2)}(Y)
:={\rm dim}_{\mathbf{G}}(\mathcal{H}^{p,q}_{(2)}(Y)).
$$
\end{defn}

\begin{rmk}
To justify the definition, note that if $Y$ is compact, the Hodge decomposition theorem implies that the Hodge number satisfies that
$$
h^{p,q}(Y)
:={\rm dim}_{\mathbb{C}}(H^{p,q}(Y))
={\rm dim}_{\mathbb{C}}(\mathcal{H}^{p,q}(Y))
={\rm Tr}(\pi^{p,q}_{Y})
=\int_Y{\rm tr}\mathcal{K}^{p,q}_Y.
$$   
Therefore if $|\mathbf{G}|$ is finite, we have $h^{p,q}_{(2)}(Y)=h^{p,q}(Y)/|\mathbf{G}|$.
In fact, even for non-compact cases we have the following $\mathbf{G}$-equivariant isomorphism by the $L^2$-Hodge decomposition (cf. Theorem 11.27 in \cite{Luck02}):
$$
H^{p,q}_{(2)}(Y)\simeq\mathcal{H}^{p,q}_{(2)}(Y),
$$
where $H^{p,q}_{(2)}(Y)$ is the reduced $L^2$-Dolbeault coholmology.
\end{rmk}

\subsection{Kazhdan's equality for compact K\"{a}hler manifolds}
Let $X$ be a K\"{a}hler manifold.
For any open subset $U\subset X$, denote by ${\rm tr}\mathcal{K}^{p,q}_{U}$ the trace form of Schwartz kernel of the orthogonal projection $\pi^{p,q}_U:A^{p,q}_{(2)}(U)\rightarrow\mathcal{H}^{p,q}_{(2)}(U)$.
Then we have the following properties.

\begin{prop}
 \label{prop: properties of kernels}
    Let $X$ be a K\"{a}hler manifold and $U\subset X$ be an open subset.
    \begin{itemize}
        \item[(a)] For any open subset $V\subset X$ containing $U$,  we have
        $$
        \int_U {\rm tr}\mathcal{K}^{p,q}_{V}\geq
        \int_U {\rm tr}\mathcal{K}^{p,q}_{X}.
        $$
        \item[(b)] Let $U_1\subset U_2\subset\cdots\subset X$ be open subsets satisfying $X=\bigcup_mU_m$. 
        Then we have
        $$
        \lim_{m\rightarrow\infty}\int_U {\rm tr}\mathcal{K}^{p,q}_{U_m}=
        \int_U {\rm tr}\mathcal{K}^{p,q}_{X}.
        $$
    \end{itemize}
\end{prop}

\begin{rmk}
    For the Bergman kernel form $\mathcal{B}_{X}:={\rm tr}\mathcal{K}^{n,0}_{X}$, (b) follows from the well-known theorem by Ramadanov.
\end{rmk}

\begin{proof}
This proposition was proved in the case of Riemann surfaces in Proposition 3.5 in \cite{BaikShokriehWu20} in terms of canonical measures.
In fact, this can be generalized to our case using the same proof.
We briefly sketch the proof for reader's convenience.

By using Tonelli's theorem and Parseval's theorem, one can prove that for any open subsets $U\subset V\subset X$, 
$$
\int_U {\rm tr}\mathcal{K}^{p,q}_{V}=\sum_{\alpha}\inner{\pi^{p,q}_V(e^U_\alpha),e^U_\alpha},
$$
where $\{e^U_\alpha\}$ is an orthonormal basis for $A^{p,q}_{(2)}(U)$ (cf. Proposition 3.4 in \cite{BaikShokriehWu20}).

Let $\mathcal{H}'(V)$ be the Hilbert subspace of $A^{p,q}_{(2)}(X)$ consisting of $L^2$-$(p,q)$-forms on $X$ which are harmonic on $V$.
Then we have
\begin{equation}\label{eqn: projection}
    \inner{\pi^{p,q}_V(e^U_\alpha),e^U_\alpha}=\inner{\pi'_V(e^U_\alpha),e^U_\alpha},
\end{equation}
where $\pi':A^{p,q}_{(2)}(X)\rightarrow\mathcal{H}'(V)$ is the orthogonal projection (see (3.4) in \cite{BaikShokriehWu20}).

Therefore, (a) follows from (\ref{eqn: projection}) and the fact that 
$$
\inner{\pi'_V(e^U_\alpha),e^U_\alpha}\geq \inner{\pi^{p,q}_X(e^U_\alpha),e^U_\alpha},
$$
since $\mathcal{H}^{p,q}_{(2)}(X)\subset \mathcal{H}'(V)$.
For (b), note that since $\mathcal{H}'(U_m)$ is a decreasing sequence of Hilbert subspaces converging to $\mathcal{H}^{p,q}_{(2)}(X)$, we have
$$
\lim_{m\rightarrow\infty}\inner{\pi'_{U_m}(e^U_\alpha),e^U_\alpha}=\inner{\pi^{p,q}_X(e^U_\alpha),e^U_\alpha},
$$
where $\pi'_{U_m}:A^{p,q}_{(2)}(X)\rightarrow\mathcal{H}'(U_m)$ is the orthogonal projection.
Then the result follows from (\ref{eqn: projection}) (for $U_m$ instead of $V$).
\end{proof}

Let $\{\phi_j:X_j\rightarrow X\}$ be a tower of coverings of a compact K\"{a}hler manifold $X$ converging to an infinite Galois covering $\phi:\widetilde{X}\rightarrow X$. 
Denote their $L^2$-Hodge numbers by $h^{p,q}_{(2)}(X_j)$ and $h^{p,q}_{(2)}(\widetilde{X})$, respectively.

\begin{prop}\label{prop: Kazhdan inequality}
We have the following Kazhdan type inequality:
\begin{equation} \label{Kazhdan}
\limsup\limits_{j\rightarrow\infty}h^{p,q}_{(2)}(X_j)\leq
h^{p,q}_{(2)}(\widetilde{X}).
\end{equation}
\end{prop}

\begin{proof}
    Let $\widetilde{F}$ be a fundamental domain of the covering $\phi:\widetilde{X}\rightarrow X$, and $\{\widetilde{F_j}\}_{j=1}$ be an increasing sequence of fundamental domains of the coverings $\{\widetilde{\phi_j}:\widetilde{X}\rightarrow X_j\}_{j=1}$.
    Let $U_1\subset U_2\subset\cdots\subset \widetilde{X}$ be a sequence of open subsets containing $\widetilde{F}$ satisfying $\widetilde{X}=\bigcup_mU_m$.
    Note that each $U_m$ is contained in all but finitely many of $\widetilde{F_j}$.
    
    By Proposition \ref{prop: properties of kernels} (a),
    $$       
    \limsup\limits_{j\rightarrow\infty} \int_{\widetilde{F}} {\rm tr}\mathcal{K}^{p,q}_{\widetilde{F_j}}\leq
    \int_{\widetilde{F}} {\rm tr}\mathcal{K}^{p,q}_{U_m}.
    $$
    Let $F_j$ be a fundamental domain of the covering $\phi_j:X_j\rightarrow X$.
    Then we have
    $$
    \int_{\widetilde{F}} {\rm tr}\mathcal{K}^{p,q}_{\widetilde{F_j}}
    =\int_{F_j} {\rm tr}\mathcal{K}^{p,q}_{X_j}
    ={\rm dim}_{\mathbf{G}_j}(\mathcal{H}^{p,q}_{(2)}(X_j))
    =h^{p,q}_{(2)}(X_j).
    $$
    On the other hand, Proposition \ref{prop: properties of kernels} (b) implies that
    $$
    \lim_{m\rightarrow\infty}\int_{\widetilde{F}} {\rm tr}\mathcal{K}^{p,q}_{U_m}
    =\int_{\widetilde{F}} {\rm tr}\mathcal{K}^{p,q}_{\widetilde{X}}
    ={\rm dim}_\mathbf{G}(\mathcal{H}^{p,q}_{(2)}(\widetilde{X}))
    =h^{p,q}_{(2)}(\widetilde{X}).
    $$
\end{proof}


\begin{proof}[Proof of Theorem \ref{thm B}]
From Theorem \ref{thm: top Kazhdan eq} and \ref{thm: L2 hodge-de rham}, we have
$$
    \lim_{j \rightarrow \infty}{\rm dim}_{\mathbf{G}_j} (\mathcal{H}^k_{(2)}(X_j)) = 
    {\rm dim}_{\mathbf{G}} (\mathcal{H}^k_{(2)}(\widetilde{X})).   
$$
Now, the two decompositions 
$$
    \mathcal{H}^{k}_{(2)}(X_j) = \bigoplus_{p+q=k}\mathcal{H}^{p,q}_{(2)}(X_j), \quad
    \mathcal{H}^{k}_{(2)}(\widetilde{X}) = \bigoplus_{p+q=k}\mathcal{H}^{p,q}_{(2)}(\widetilde{X}),
$$
and Remark 4.2 (iv) in \cite{BaikShokriehWu20} imply that
\begin{equation} \label{Kazhdan2}
 \sum\limits_{p+q=k} h^{p,q}_{(2)}(X_j) =
 \sum\limits_{p+q=k}{\rm dim}_\mathbf{G_j}(\mathcal{H}^{p,q}_{(2)}(X_j)) \rightarrow
\sum\limits_{p+q=k}{\rm dim}_\mathbf{G}(\mathcal{H}^{p,q}_{(2)}(\widetilde{X})) = \sum\limits_{p+q=k} h^{p,q}_{(2)}(\widetilde{X}),
\end{equation}
as $j \rightarrow \infty$. 
Therefore, (\ref{Kazhdan2}) and Proposition \ref{prop: Kazhdan inequality} yield that
$$
\limsup\limits_{j\rightarrow\infty}h^{p,q}_{(2)}(X_j)=
h^{p,q}_{(2)}(\widetilde{X}).
$$
Moreover, the elementary properties of $\liminf$ and $\limsup$:
$$
\limsup\limits_{j\rightarrow\infty}(a_j+b_j)\leq\limsup\limits_{j\rightarrow\infty}(a_j)+\limsup\limits_{j\rightarrow\infty}(b_j),
$$
$$
\liminf\limits_{j\rightarrow\infty}(a_j+b_j)\leq \liminf\limits_{j\rightarrow\infty}(a_j)+\limsup\limits_{j\rightarrow\infty}(b_j)\leq\limsup\limits_{j\rightarrow\infty}(a_j+b_j)
$$
imply that
$$
\liminf\limits_{j\rightarrow\infty}h^{p,q}_{(2)}(X_j)
=h^{p,q}_{(2)}(\widetilde{X})
=\limsup\limits_{j\rightarrow\infty}h^{p,q}_{(2)}(X_j),
$$
which completes the proof.

\end{proof}

\section{limits of the Bergman kernels and metrics} \label{sec 5}



In the previous section, we proved the convergence of $L^2$-Hodge numbers 
$ ( h^{p,q}_{(2)}(X_j) := \int_{F_j} {\rm tr}\mathcal{K}^{p,q}_{X_j} )$ on a tower of coverings $\{ \phi_j : X_j \rightarrow X \}$, where $F_j$ is a fundamental domain of $\phi_j$. In this section, we show the (local) uniform convergence of the trace forms of Schwartz kernels (of $(n,0)$ type) $  {\rm tr}\mathcal{K}^{n,0}_{X_j}$, which is also called the Bergman kernel.

For a complex manifold $X$, let $\Omega^{n,0}_{(2)}(X)$ be the space of square integrable holomorphic $(n,0)$-forms equipped with the inner product
$\inner{f,g} = i^{n^2} 2^{-n} \int_{X} f \wedge \overline{g}$, and the norm is defined by $\norm{f}^2 := \inner{f,f}. $ 
Then $\Omega^{n,0}_{(2)}(X)$ is a Hilbert space because it is a closed subspace of $A^{n,0}_{(2)}(X)$.
The {\it Bergman kernel} on $X \times X$ is defined by 
$$ \Bergman_X(z,w) := \sum_{\alpha = 0} s_{\alpha}(z) \wedge \overline{s_{\alpha}(w)} ,$$
where $\{ s_{\alpha} \}$ is an orthonormal basis for $\Omega^{n,0}_{(2)}(X)$. 
Note that since $\Omega^{(n,0)}_{(2)}(X) =  \mathcal{H}^{n,0}_{(2)}(X)$, $\Bergman_X(z,z) = {\rm tr}\mathcal{K}^{n,0}_X$ (see Definition \ref{def: Schwartz kernel of (p,q)}).

For a holomorphic local coordinate chart $(z_1, \cdots,  z_n, w_1, \cdots, w_n)$ on $X \times X$, 
denote by $dz = dz_1 \wedge \cdots \wedge dz_n$, $d\overline{z} = d\overline{z}_1 \wedge \cdots \wedge d\overline{z}_n$, $dw = dw_1 \wedge \cdots \wedge dw_n$, and $d\overline{w} = d\overline{w}_1 \wedge \cdots \wedge d\overline{w}_n$.
Then $\Bergman_X(z,w)$ can be written on the chart as 
$$ \Bergman_{X}(z,w) = \Bergman_{X}^*(z,w)  dz \wedge d\overline{w}, $$
and $\Bergman_{X}^*(z,w)$ is called the {\it Bergman kernel function}. 
When $\Bergman_X^*(z) := \Bergman_X^*(z,z) > 0$, the {\it Bergman (pseudo-)metric} on $X$ is defined by

$$ ds^2_X := \sum_{j,k=1}^n \frac{\partial^2 \log \Bergman^*_X(z)}{\partial z_j \partial \overline{z}_k } dz_j \otimes d\overline{z}_k . $$ 

\begin{rmk} \label{remark on kernel}
    \begin{enumerate}
        \item $\Bergman_X(z,w)$ is independent of the choice of the orthonormal basis.
        \item \label{well-defineness of push-forward kernel} 
            The Bergman kernel on the diagonal $\Bergman_X(z)$ is a biholomorphic invariant $(n,n)$-form on $X$.
            Hence, for a covering $\phi : Y \rightarrow Y / \mathbf{G}$ where $\mathbf{G} \subset \Aut(Y)$, the push-forward form $\phi_* \Bergman_Y(z)$ is well-defined on $Y / \mathbf{G}$.
        \item \label{reproducing property}
            For $y \in X$, fix a holomorphic local chart $(w_1, \cdots, w_n)$ near $y$. Then, by the Riesz representation theorem, 
            there exists $\Bergman^y_X \in \Omega^{n,0}_{(2)}(X)$ such that 
                \begin{equation*} 
                    f^*(y) = \inner{f, \Bergman^y_X}   
                \end{equation*}
                for all $f \in \Omega^{n,0}_{(2)}(X)$, where $f(w)=f^*(w)dw$ on the chart.
                This is called the reproducing property. If $f=\Bergman_X^y$, then
                $$ \Bergman_X^*(y,y) = \norm{\Bergman_X^y}^2 ,$$
                where $\Bergman_X(w,w) = \Bergman_X^*(w,w)dw\wedge d\overline{w}$.
                In terms of an orthonormal basis $\{ s_{\alpha} \}$, $\Bergman^y_X$ can be written as
                $$ \Bergman^y_X(z) = \sum_{\alpha=0} \overline{s_{\alpha}^*(y)}  s_{\alpha}(z)  ,$$
                where $s_{\alpha}(w) = s_{\alpha}^*(w)dw$. 
        \item For a fixed point $x \in X$, $\Bergman_X(x)$ satisfies the following extremal property:
            $$ \Bergman_X(x) = \max \{f(x) \wedge \overline{f(x)} \:  | \:  f \in \Omega^{n,0}_{(2)}(X), \norm{f}=1 \} .$$
        \item $\Bergman_{X}^*(z,w)$ is holomorphic in $z$ variables and anti-holomorphic in $w$ variables. 
    \end{enumerate}
\end{rmk}

\begin{rmk} \label{remark on metric}
    \begin{enumerate}
        \item Although $\Bergman_X^*(z)$ depends on a coordinate chart, $ds^2_X$ is independent of the choice of the chart.
        \item \label{well-defineness of push-forward metric} 
            The Bergman metric is invariant under biholomorphisms, i.e., $ds^2_Y = F^* ds^2_X$ for a biholomorphic map $F : Y \rightarrow X$. 
            Hence, for a covering $\phi : Y \rightarrow Y / \mathbf{G}$ where $\mathbf{G} \subset \Aut(Y)$, the push-forward metric $\phi_* ds^2_Y$ is well-defined on $Y / \mathbf{G}$.
    \end{enumerate}
\end{rmk}

Let $\{\phi_j : X_j \rightarrow X\}$ be a tower of coverings of a compact complex manifold $X$ converging to an infinite Galois covering $\phi : \widetilde{X} \rightarrow X$, and $\{\widetilde{\phi}_j : \widetilde{X} \rightarrow X_j\}$ be a sequence of coverings such that $\phi = \phi_j \circ \widetilde{\phi}_j$.

\begin{prop} \label{prop: uniform convergence of Bergman kernels}
    Assume that
	\begin{equation} \label{int converges}
		\int_{X} \phi_{j*} \Bergman_{X_j}(w,w) \rightarrow \int_{X} \phi_* \Bergman_{\widetilde{X}}(w,w).
	\end{equation}
    Then
	\begin{equation}
	    \widetilde{\phi}_j^* \Bergman_{X_j}(z,w) \rightarrow \Bergman_{\widetilde{X}}(z,w) \text{ locally in } C^{\infty} \text{-topology on } \widetilde{X} \times \widetilde{X}.
	\end{equation}
\end{prop}

\begin{proof}
    For $x,y \in \widetilde{X}$, choose two holomorphic local charts $z : U_{x} \rightarrow P(0,r) \subset \CC^n$ and $w : U_{y} \rightarrow P(0,r) \subset \CC^n$, where $U_x, U_y$ are open neighborhoods of $x$ and $y$ in $\widetilde{X}$, respectively, and $P(0,r)$ is the polydisc of radius $r$ centered at $0$. 
    Then the Bergman kernels of $\Bergman_{\widetilde{X}}$ and $\widetilde{\phi}_j^* \Bergman_{X_j}$ on that charts are given by
    \begin{align*}
        \Bergman_{\widetilde{X}}(z, w) &= \Bergman_{\widetilde{X}}^*(z, w)  dz \wedge d\overline{w}, \\
        \widetilde{\phi}_j^* \Bergman_{X_j}(z, w) &= (\widetilde{\phi}_j^* \Bergman_{X_j})^*(z, w)  dz \wedge d\overline{w}.
    \end{align*}
    Since $\Bergman_{\widetilde{X}}^*(z, w)$ and $(\widetilde{\phi}_j^* \Bergman_{X_j})^*(z, w)$ are holomorphic in $z$ and anti-holomorphic in $w$, we may apply the Cauchy estimate in both variables. Then, for $\alpha, \beta \in \NN^n$,
    \begin{align*}
        & | \partial^{\overline{\beta}} \partial^{\alpha} \Bergman_{\widetilde{X}}^*(x,y) - \partial^{\overline{\beta}} \partial^{\alpha} (\widetilde{\phi}_j^* \Bergman_{X_j})^*(x,y) |^2 \\
        & \le C' \int_{P(0,r/2)} | \partial^{\alpha} \Bergman_{\widetilde{X}}^*(x, w) - \partial^{\alpha} (\widetilde{\phi}_j^* \Bergman_{X_j})^*(x, w) |^2   dwd\overline{w} \\  
        & \le C \int_{P(0,r/2)} \int_{P(0,r/2)} |\Bergman_{\widetilde{X}}^*(z, w) -  (\widetilde{\phi}_j^* \Bergman_{X_j})^*(z, w)|^2 dzd\overline{z} dwd\overline{w} \\
        & \le C \int_{P(0,r/2)} \norm{ \Bergman_{\widetilde{X}}^w - \Bergman_{\widetilde{F}_{j}}^w }_{\widetilde{F}_{j}}^2 
        + \norm{ \Bergman_{\widetilde{F}_{j}}^w - \widetilde{\phi}_j^* \Bergman_{X_j}^w }_{\widetilde{F}_{j}}^2  dwd\overline{w},
    \end{align*}
    where $C$ depends only on $\alpha, \beta$ and $r$, $\widetilde{F}_j$ is a fundamental domain of $\widetilde{\phi}_j : \widetilde{X} \rightarrow X_j$ satisfying $U_x, U_y \subset \widetilde{F}_j$ (by increasing $j$ if necessary), and $\Bergman_{\widetilde{F}_{j}}^w = \Bergman_{\widetilde{F}_{j}}( \cdot, w)$ is the Bergman kernel of $\widetilde{F}_{j}$.
    Now, for $w \in \widetilde{F}_{j}$,
    \begin{align*}
        & \norm{ \Bergman_{\widetilde{X}}^w - \Bergman_{\widetilde{F}_{j}}^w}_{\widetilde{F}_{j}}^2 \\
        & = \norm{\Bergman_{\widetilde{X}_{\phantom{j}}}^w}_{\widetilde{F}_{j}}^2 + \norm{\Bergman_{\widetilde{F}_{j}}^w}_{\widetilde{F}_{j}}^2 
        - 2\re \inner{\Bergman_{\widetilde{X}}^w, \Bergman_{\widetilde{F}_{j}}^w}_{\widetilde{F}_{j}} \\
        & \le \norm{\Bergman_{\widetilde{X}_{\phantom{j}}}^w}_{\widetilde{X}}^2 + \norm{\Bergman_{\widetilde{F}_{j}}^w}_{\widetilde{F}_{j}}^2 
        - 2\re \inner{\Bergman_{\widetilde{X}}^w, \Bergman_{\widetilde{F}_{j}}^w}_{\widetilde{F}_{j}} \\
        & = \Bergman_{\widetilde{X}}^*(w,w) + \Bergman_{\widetilde{F}_{j}}^*(w,w) - 2 \Bergman_{\widetilde{X}}^*(w,w),    
        &(\text{Remark } \ref{remark on kernel} \: (\ref{reproducing property}))    
    \end{align*}
    and similarly
    \begin{align*}
        & \norm{ \Bergman_{\widetilde{F}_{j}}^w - \widetilde{\phi}_j^*\Bergman_{X_{j}}^w }_{\widetilde{F}_{j}}^2 \\
        & = \norm{\Bergman_{\widetilde{F}_{j}}^w}_{\widetilde{F}_{j}}^2 + \norm{\widetilde{\phi}_j^*\Bergman_{X_{j}}^w}_{\widetilde{F}_{j}}^2 
        - 2\re \inner{ \Bergman_{\widetilde{F}_{j}}^w, \widetilde{\phi}_j^*\Bergman_{X_{j}}^w }_{\widetilde{F}_{j}}  \\
        & = \Bergman_{\widetilde{F}_{j}}^*(w,w) + (\widetilde{\phi}_j^*\Bergman_{X_{j}})^*(w,w) 
        - 2 (\widetilde{\phi}_j^*\Bergman_{X_{j}})^*(w,w).
        &(\text{Remark } \ref{remark on kernel} \: (\ref{reproducing property}))    
    \end{align*}
    Here, in the last equality, we used 
    $\norm{\widetilde{\phi}_j^*\Bergman_{X_{j}}^w}_{\widetilde{F}_{j}}^2 = \norm{\Bergman_{X_{j}}^w}_{\widetilde{\phi}_j(\widetilde{F}_j)}^2 = \norm{\Bergman_{X_{j}}^w}_{X_{j}}^2 $.
    Let $\widetilde{F}^y$ be a fundamental domain of $\phi : \widetilde{X} \rightarrow X$ containing $U_y$. Then altogether,
    \begin{align*}
        & | \partial^{\overline{\beta}} \partial^{\alpha} \Bergman_{\widetilde{X}}^*(x,y) - \partial^{\overline{\beta}} \partial^{\alpha} (\widetilde{\phi}_j^*\Bergman_{X_j})^*(x,y) |^2 \\
        & \le 2 C \int_{\widetilde{F}^y} \left( \Bergman_{\widetilde{F}_{j}}(w,w) - \Bergman_{\widetilde{X}}(w,w) \right)
        + C \int_{\widetilde{F}^y} \left(\Bergman_{\widetilde{X}}(w,w) - \widetilde{\phi}_j^*\Bergman_{X_{j}}(w,w) \right) .
    \end{align*}
    Note that all two terms on the right-hand side do not depend on the choice of initial points $x,y \in \widetilde{X}$, and converge to $0$ as $j \rightarrow \infty$ because of 
    Proposition \ref{prop: properties of kernels} (b), the assumption, and
    the fact that 
    $\int_{\widetilde{F}^y} \left(\Bergman_{\widetilde{X}}(w,w) - \widetilde{\phi}_j^* \Bergman_j(w,w)   \right) 
    = \int_{X} \left(\phi_* \Bergman_{\widetilde{X}}(w,w) - \phi_{j*} \Bergman_{X_{j}}(w,w)   \right)$.
    This completes the proof.
\end{proof}

\begin{thm} \label{thm: Bergman stable}
    Let $\{ \phi_j : X_j \rightarrow X \}$ be a tower of coverings of a compact K\"ahler manifold $X$ converging to an infinite Galois covering 
    $\{ \phi : \widetilde{X} \rightarrow X \}$. Then the sequence of Bergman kernels $\widetilde{\phi}_j^* \Bergman_{X_j}(z,w)$ converges to $\Bergman_{\widetilde{X}}(z, w)$ locally in $C^{\infty}$-topology on $\widetilde{X} \times \widetilde{X}$.
    Moreover, with an additional condition that $\widetilde{X}$ admits the Bergman metric, the sequence of the Bergman metrics $\widetilde{\phi}_j^* ds^2_{X_j}$ converges to $ ds^2_{\widetilde{X}}$ locally in $C^{\infty}$-topology on $\widetilde{X}$.
\end{thm}

\begin{rmk}
    By Corollary \ref{thm C}, $ds^2_{\widetilde{X}} > 0$ implies that $ds^2_{X_j} > 0$ for all $j \ge N$ for some $N \in \NN$. 
\end{rmk}

\begin{proof}
    First note that, for any $n$-dimensional complex manifold $Y$, $\mathcal{H}^{n,0}_{(2)}(Y) = \Omega^{n,0}_{(2)}(Y) $ because $\Delta_{\overline{{\partial}}} u=0$ if and only if $\overline{\partial} u = 0$ and $\overline{\partial}^* u = 0$.
    Hence Theorem \ref{thm B} implies that $\int_{X} \phi_{j*} \Bergman_{X_j}(w,w) \rightarrow \int_{X} \phi_* \Bergman_{\widetilde{X}}(w,w)$, which shows the convergence of the Bergman kernels by Proposition \ref{prop: uniform convergence of Bergman kernels}. 
    The convergence of the Bergman metrics follows from the local expression 
    $$ 
    \partial\overline{\partial} \log \Bergman^* = \frac{\partial\overline{\partial} \Bergman^*}{\Bergman^*} - \frac{(\partial\Bergman^*) (\overline{\partial}\Bergman^*)}{(\Bergman^*)^2}
    $$
    and the convergence of the Bergman kernels.
\end{proof}


\section{Immersion into a projective space} \label{sec 6}

Let $K_X$ be the canonical line bundle over a complex manifold $X$ (not necessarily compact).
Let $H^0_{(2)}(X, K_X)$ be the Hilbert space of all holomorphic $L^2$-sections of $K_X$, i.e.,  $H^0_{(2)}(X, K_X)= \Omega_{(2)}^{n,0}(X)$,
and $\PP( H^0_{(2)}(X, K_X)^*)$ be the projectivization of its dual space $H^0_{(2)}(X, K_X)^*$.
Then there exists a canonical map $\iota_X : X \rightarrow \PP( H^0_{(2)}(X, K_X)^*) \cong \CC\PP^N (N \le \infty)$ defined by $x \mapsto \left[ s_0(x), \cdots, s_N(x) \right]$, where $\{s_{\alpha}\}_{\alpha=0}$ is a basis for $H^0_{(2)}(X, K_X)$.
We say that $K_X$ is {\it very ample} if the map $\iota_X$ is an embedding.

In this section, we consider the following immersion problem; 
if a canonical map $\widetilde{X} \rightarrow \PP(H^0_{(2)}(\widetilde{X}, K_{\widetilde{X}})^*)$ is an immersion, then a canonical map $X_j \rightarrow \PP(H^0(X_j, K_{X_j})^*)$ is also an immersion for sufficiently large $j$.
Note that whether a canonical map is an immersion or not, as well as very ampleness, is independent of the choice of the basis,
hence one may assume $\{s_{\alpha}\}_{\alpha=0}$ is an orthonormal basis.
To show that $\iota_X$ is an immersion, we need to show 
$K_X$ is base-point free (i.e., for all $x \in X$ there exists a holomorphic section $s \in H^0_{(2)}(X, K_X)$ such that $s(x) \neq 0$), 
and $d\iota_X$ is injective.
The following two lemmas are based on Kobayashi (see Theorem 8.1 and 8.2 in \cite{Kobayashi59}).

\begin{lem} \label{lem: base-point free}
    $K_X$ is base-point free if and only if 
    the Bergman kernel $\Bergman_X(x,x) > 0$ for all $x \in X$.
\end{lem}
\begin{proof}
    It easily follows from the fact $ \Bergman_X(x,x) = \sum_{\alpha=0}^{N} s_{\alpha}(x) \overline{s_{\alpha}(x)} $.
\end{proof}

\begin{lem} \label{lem: immersion}
    Assume $K_X$ is base-point free and $\{s_{\alpha}\}_{\alpha=0}$ is an orthonormal basis. Then
    $d\iota_X$ is injective if and only if
    $X$ admits the Bergman metric (i.e., $ds^2_X = \IM \partial \overline{\partial} \log \Bergman_X^* > 0$).
\end{lem}
\begin{proof}
    For $x \in X$ and a non-zero vector $V \in T^{1,0}_x(X)$, we first choose an orthonormal basis $\{ s_{\alpha} \}_{\alpha=0}$ such that 
    $s_0(x) \neq 0$, $s_{\alpha}(x)=0$ for all $\alpha \ge 1$, and $Vs_1(x) \neq 0$, $Vs_{\alpha}(x) = 0$ for all $\alpha \ge 2$. 
    This is possible because the linear functionals $S : H^0_{(2)}(X, K_X) \rightarrow \CC$ and $S' : {\rm Ker}(S) \rightarrow \CC$ given by $f \mapsto f(x)$ and $f \mapsto Vf(x)$, respectively, are continuous.
    Then 
    \begin{align*}
        ds^2_X(V,V) &= \frac{\left(|Vs_0(x)|^2 + |Vs_1(x)|^2 \right) |s_0(x)|^2 - \left(s_0(x) \overline{Vs_0(x)}\right)\left(Vs_0(x) \overline{s_0(x)}\right)}{|s_0(x)|^4} \\
        &= \frac{|Vs_1(x)|^2}{|s_0(x)|^2}.
    \end{align*}
    Hence, $ds^2_X(V,V) > 0$ if and only if $Vs_1(x) \neq 0$.
    On the other hand, for $\iota_X = ( \frac{s_1}{s_0}, \frac{s_2}{s_0}, \cdots )$ with the same special orthonormal basis $\{ s_{\alpha} \}_{\alpha=0}$ as above,
    \begin{equation*}
        d\iota_X(V) = \left( \frac{Vs_1(x)}{s_0(x)}, 0, 0,  \cdots  \right).
    \end{equation*}
    Hence, $d\iota_X(V) \neq 0$ if and only if $Vs_1(x) \neq 0$.
\end{proof}

\begin{rmk} \label{rmk: separation points}
    There is also a criterion for the separation of points (i.e., for all $x, y \in X$ there exists a holomorphic section $s \in \Gamma_{(2)}(X, K_X)$ such that $s(x) = 0$ and $s(y) \neq 0$) in terms of the Bergman kernel. In \cite{Calabi53}, E. Calabi gave necessary and sufficient conditions for the isometric embeddings of complex manifolds in terms of the so-called {\it diastasis} function $D$.
    In case of the Bergman metric, the diastasis function is given by
    $$ D_X(x,y) = \log \frac{\Bergman^*_X(x,x) \Bergman^*_X(y,y)}{|\Bergman^*_X(x,y)|^2}  .$$
    X. Wang \cite{Wang15} showed that $K_X$ separates points if and only if $D_X(x,y) > 0$ for $x \neq y \in X$, provided that $\Bergman^*_X(x,y) \neq 0$, which is one of Calabi's conditions for isometric embeddings. 
\end{rmk}

\begin{proof}[Proof of Corollary \ref{thm C}]
    By Lemma \ref{lem: base-point free} and \ref{lem: immersion}, since a canonical map $\widetilde{X} \rightarrow \PP(H^0_{(2)}(\widetilde{X}, K_{\widetilde{X}})^*)$ is an immersion, $\phi_* \Bergman_{\widetilde{X}} (x,x) > 0$ and $\phi_* ds^2_{\widetilde{X}} > 0$ on $X$.
    Now, since $X$ is compact, the convergence of the Bergman kernels (Theorem \ref{thm: Bergman stable}) implies that there exists $N \in \NN$ such that 
    $\phi_{j*} \Bergman_{X_j} (x,x) > 0$ and $\phi_{j*} ds^2_{X_j} > 0$ on $X$ for all $j \ge N$.
    Therefore, we conclude that there exists an immersion $X_j \rightarrow \PP(H^0(X_j, K_{X_j})^*)$ for all $j \ge N$ by applying the above two lemmas again.  
\end{proof}

    Since we have a good criterion for the separation of points (Remark \ref{rmk: separation points}) and the convergence of the off-diagonal Bergman kernels (Theorem \ref{thm: Bergman stable}), it seems that, by using the same argument as above, one can prove the following statement;
    if $K_{\widetilde{X}}$ is very ample, then $K_{X_j}$ is also very ample for sufficiently large $j$.
    However, the above argument fails because we cannot push forward $\Bergman_{X_j}(x,y)$ (hence $D_{X_j}(x,y)$) to the base manifold $X$, which is compact. On the other hand, S.-K. Yeung \cite{Yeung00} showed that $K_{X_j}$ is also very ample for sufficiently large $j$ under the assumption that either the top manifold $\widetilde{X}$ is a Hermitian symmetric manifold of non-compact type or the base manifold $X$ is a K\"ahler manifold with negatively pinched sectional curvature. 
    In both cases, there is no $L^2$-harmonic forms on the universal covering $\widetilde{X}$ except for $L^2$-holomorphic $n$-forms which form an infinite-dimensional vector space. He (\cite{Yeung01}, \cite{Yeung17}) used the above fact crucially together with the heat kernel estimates based on Donnelly's results (\cite{Donnelly79}).

\bibliographystyle{abbrv}
\bibliography{reference}

\end{document}